\documentclass[12pt,titlepage]{article}
\usepackage{hyperref}

\usepackage{setspace}
\usepackage[margin=3cm]{geometry}

\usepackage{pstricks}
\psset{linewidth=0.5pt,arrowlength=2.8}

\newtheorem{theorem}{Theorem}
\newtheorem{lemma}{Lemma}
\newtheorem{corollary}{Corollary}

\newenvironment{proof}{\noindent{\em Proof.}{}}{\hfill\rule{0.2em}{0.5em}\medskip}

\newcommand{\ZZ}{\mathbf{Z}}
\newcommand{\RR}{\mathbf{R}}
\newcommand{\conv}{\mathop{\rm conv}}
\newcommand{\cone}{\mathop{\rm cone}}
\newcommand{\Ver}{\mathop{\rm Vert}}
\newcommand{\TT}{{\cal T}}

\begin{document}

\title{On the number of irreducible points in~polyhedra}
\author{A.\,Yu.\,Chirkov \qquad N.\,Yu.\,Zolotykh}
\date{chirkov@vmk.unn.ru \qquad \underline{zolotykh@vmk.unn.ru}  \\
Lobachevsky State University of Nizhni Novgorod\\
pr. Gagarina, 23; Nizhni Novgorod; 603950 Russia\\[2em]
{\it Abbreviated title:} On the number of irreducible points\\[2em]
}
\maketitle

\begin{abstract}
An integer point in a polyhedron is called irreducible iff it is not the midpoint of two other integer points in the polyhedron.
We prove that the number of irreducible integer points in $n$-dimensional polytope with radius $k$ given by a system of $m$ linear inequalities is at most $O(m^{\lfloor\frac{n}{2}\rfloor}\log^{n-1} k)$ if $n$ is fixed.
Using this result we prove the hypothesis asserting that the teaching dimension in the class of threshold functions of $k$-valued logic in $n$ variables is
$\Theta(\log^{n-2} k)$ for any fixed $n\ge 2$.
\end{abstract}


\section{Introduction}

In the paper we study the number of irreducible points in polyhedra.
Let $P$ be a (convex) polyhedron in $\RR^n$. 
A point $x\in P\cap\ZZ^n$ is called to be {\em irreducible} in $P$ (more precisely, in $P\cap\ZZ^n$)
iff $x$ can not be represented as $x=\frac{1}{2}(y+z)$, where $y$ and $z$ are different points in $P\cap\ZZ^n$. 

It is not hard to see that any vertex of $P_I=\conv(P\cap\ZZ^n)$ is irreducible in $P\cap\ZZ^n$.
The converse is not true, as the following example shows. 
Let $P=\{x\in\RR^2:\ x_1+x_2\ge 1,\ 2x_1-x_2\le 2,\ -x_1+2x_2\le 2\}$.
The point $(1,1)$ is irreducible in $P$, but it is not a vertex of $P_I$.
Nevertheless these properties (irreducibility and being vertex) are similar, as evidenced by a nearness of bounds for
the number of vertices and the number of irreducible points \cite{Shevchenko1981}.
We remark that the number of vertices in $P_I$ is more studied than the number of irreducible points
because of the importance of the former in integer linear programming, but the number of irreducible points is of separate interest too.
In particular, they appear in studying a teaching dimension of the class of threshold functions (see Section~\ref{sec_td_irr}).

Suppose that a polyhedron $P$ is given by a system of linear inequalities $Ax\le b$, where $A=(a_{ij})\in\ZZ^{m\times n}$, $b=(b_i)\in\ZZ^m$, $|a_{ij}| \le \alpha$, $|b_i|\le\beta$, $\gamma = \max\{\alpha,\,\beta\}$.
Let $\varphi$ be the sum of the sizes of all inequalities in $Ax\le b$, that is, $\varphi = O(mn\log\gamma)$.

Upper bounds for the number of vertices in $P_I$ are proposed in \cite{Shevchenko1981}, \cite{HayesLarman1983}, \cite{Morgan1991}, \cite{CHKMcD1992}, etc.
In \cite{CHKMcD1992} it is proved that, for any fixed $n$, $P_I$ has at most $O(m^n \varphi^{n-1})$ vertices.
More precise bound, $O(m^{\lfloor \frac{n}{2}\rfloor}\log^{n-1}\gamma)$, was found in~\cite{Chirkov1997} (see~\cite{VeselovChirkov2007}).
To obtain the upper bounds an approach due to Shevchenko \cite{Shevchenko1981} based on a {\em separation property} is used.
A method developed in \cite{HayesLarman1983} and \cite{CHKMcD1992} using {\em reflecting sets} is essentially equivalent to Shevchenko's approach.
In fact, the results in \cite{Shevchenko1981} and \cite{HayesLarman1983} don't use any facts about vertices of $P_I$ other than they are not
irreducible. Thus, the upper bounds in these two earliest papers are valid for the number of irreducible points too. 
On the other hand, the proofs of tight bounds in~\cite{CHKMcD1992} and \cite{Chirkov1997} use the irreducibility of the vertices as well as
their extremality. So, these proofs are not applicable for the case of the number of irreducible points.

Lower bounds for the number of vertices in $P_I$ are proposed in \cite{Veselov1984}, \cite{Morgan1991}, \cite{BHL1992}, \cite{Chirkov1996}, etc.
In \cite{Veselov1984} it was shown that, for any fixed $n$, a knapsack polytope can have $\Omega(\varphi^{n-1})$ vertices.
Another construction with $\Omega(\varphi^{n-1})$ lower bound was proposed in \cite{BHL1992}.
In \cite{Chirkov1996} it was proved that, for any $m$ and any fixed $n$, there exists a polyhedron with
$\Omega(m^{\lfloor \frac{n}{2}\rfloor}\log^{n-1}\gamma)$ vertices.
It is clear that the lower bounds for the number of vertices in $P_I$ are also true for the number of irreducible points in $P$. 

For some other results and comments concerning bounds for the number of vertices in $P_I$
see \cite{Shevchenko1995}, \cite{Zolotykh2006}, \cite{VeselovChirkov2007}.

The main result of the paper is a tight upper bound for the number of irreducible points in a polyhedron.
When $n$ is fixed this bound is close to the lower bound and it is asymptotically
the same as the bound for the number of vertices in $P_I$.

Let $P$, $P_1, P_2, \dots, P_s$ be polytopes (bounded polyhedra) in $\RR^n$.
If $P = \cup_{i=1}^s P_i$, then $\{P_1, P_2, \dots, P_s\}$ is called an ({\em inner}) {\em cover} of the polytope $P$.
If the intersection of any two polytopes in the cover is empty or it is their common face, then
the cover is called a {\em regular partition}.
If all polytopes in a regular partition are simplexes then the partition is called a {\em triangulation}. 

Our method for finding the upper bounds for the number of irreducible points in a polytope consists of the following.
First, we obtain a bound for the number of irreducible points in a parallelepiped (see Section~\ref{sec_parallelepiped_irr}).
In Section \ref{sec_cobering_irr} we construct a cover of a polytope $P$ by parallelepipeds $P_1, P_2, \dots, P_s$.
To do this we build a triangulation of the polytope, then each simplex in the triangulation is covered by parallelepipeds.
It is not hard to see that if $x\in P_i$ is irreducible in $P$ then $x$ is irreducible in $P_i$. 
This property allows us (in Section \ref{sec_polytope_irr}) to find a bound for the number of irreducible points in $P$.
Namely Theorem~\ref{tNirred0} asserts that the number of irreducible points in a polytope $P$ is at most
$O\left(m^{\lfloor\frac{n}{2}\rfloor}\log^{n-1} (\alpha\beta)\right)$ ($n$ is fixed).
Theorem~\ref{tNirred} asserts that if $P$ has radius $k$ then the number of its irreducible points is at most
$O(m^{\lfloor\frac{n}{2}\rfloor}\log^{n-1} k)$ ($n$ is fixed).

As in the cited papers, where the upper bounds for the number of vertices in $P_I$ were found, 
the (analogue of) separation property plays the central role in our construction too.
On the other hand, to our knowledge, using the covering of a simplex with parallelepipeds and bounding 
the number of irreducible points (or vertices) in them
are novel for the problems under consideration. Also, we believe that our construction is clearer than the methods in~\cite{CHKMcD1992} and~\cite{Chirkov1997}.

Our results on the number of irreducible points are used to prove a hypothesis concerning the teaching dimension of the class of threshold functions of $k$-valued logic in Section~\ref{sec_td_irr}.

Let $n\ge 1$, $k\ge 2$, $E_k=\{0,1,\dots,k-1\}$.
A function $f:~ E_k^n\to\{0,1\}$ is called a {\em threshold} function iff there exists a hyperplane
separating the set $M_0(f)$ of points, in which $f$ is $0$, and the set $M_1(f)$ of points, in which $f$ is $1$, that is,
there are real numbers $a_0, a_1, \dots, a_n$, such that
$$
M_0(f) = \left\{x=(x_1,x_2,\dots,x_n)\in E_k^n:~ \sum_{j=1}^n a_jx_j \le a_0 \right\}.
$$
The inequality $\sum_{j=1}^n a_jx_j \le a_0$ is called the {\em threshold inequality}.
It is easy to see that its coefficients can be chosen integer.
Denote by $\TT(n,k)$ the set of all threshold functions defined on $E_k^n$.

A set $T$ is called a {\em teaching} set for $f\in\TT(n,k)$, 
iff, for any function $g\in\TT(n,k) \setminus\{f\}$, there is a point $z\in T$ such that $f(z)\ne g(z)$.
Teaching set is appeared in the problem of deciphering threshold function (or ``learning halfspaces'' in Algorithmic Learning Theory terminology)
\cite{ZolotykhShevchenko1995}, \cite{Hegedus1995}.
A teaching set $T$ of $f\in\TT(n,k)$ is called a {\em minimal teaching} set,
if no proper subset of it is teaching for $f$.
It is known (see, for example, \cite{ShevchenkoZolotykh1998}, \cite{ZolotykhShevchenko1999}), 
that the minimal teaching set $T(f)$ of any threshold function $f$ is unique.
The maximum cardinality of minimal teaching set,
$$
\sigma(n,k) = \max_{f\in\TT(n,k)} |T(f)|,
$$
is called the {\em teaching dimension}.

Bounds for $\sigma(n, k)$ are constructed in \cite{Hegedus1994}, \cite{ShevchenkoZolotykh1998}, \cite{ZolotykhShevchenko1999}, \cite{Zolotykh2008}, \cite{ZolotykhChirkov2012}, etc.
A generalization is considered in \cite{ShevchenkoZolotykh1998a}.
It is known that $\sigma(n, k)$ depends on $n$ exponentially, in particularly, $\sigma(n,2)=2^n$.
In \cite{Hegedus1994} it is proved using \cite{Shevchenko1985}, \cite{CHKMcD1992} that for any fixed $n$
$\sigma(n, k) = O(\log^{n-1} k).$
In \cite{ShevchenkoZolotykh1998}, \cite{ZolotykhShevchenko1999} a lower bound
$\sigma(n, k) = \Omega(\log^{n-2} k)$
is obtained.
See also \cite{Zolotykh2008}. 
In \cite{ShevchenkoZolotykh1998} it is proved that $\sigma(2,k)=4$.
In \cite{ZolotykhChirkov2012} 
the following hypothesis puts forward
$$
\sigma(n, k) = \Theta(\log^{n-2} k) 
$$
for any fixed $n\ge 2$.
In Section \ref{sec_td_irr} we prove Theorem \ref{th_upper_sigma} asserting that $\sigma(n, k) = O(\log^{n-2} k)$ for any fixed $n\ge 2$.
Thus, the hypothesis is true.

Note that the average cardinality of the minimal teaching set is studied in \cite{ABS1995}, \cite{VirovlyanskayaZolotykh2003}.

\paragraph{Notation}
Suppose $X\subseteq \RR^n$.
Let $\conv X$ be the convex hull of $X$,
$\cone X$ is the cone hull of $X$ (the set of all linear combinations of vectors in $X$ with nonnegative coefficients),
$\Ver X$ is the set of vertices of $X$.
If $X$, $Y$ are subsets of $\RR^n$ then by
$X+Y$ and $X-Y$ we mean the set of all vectors of the form $x+y$ and consequently $x-y$, where $x\in X$, $y\in Y$.

\section{Irreducible points in parallelepipeds\label{sec_parallelepiped_irr}}

In this section we propose upper bounds for the number of irreducible points in a parallelepiped.

Let $A\in \ZZ^{n\times n}$ be a non-singular matrix, $c=(c_1,c_2,\dots,c_n)\in\ZZ^n$, $b=(b_1,b_2,\dots,b_n)\in\ZZ^n$,
$b< c$. 
Consider a parallelepiped 
$$
P(A,b,c)=\{x\in \RR^n:~ b\le Ax\le c\}.
$$
Denote $M(A,b,c)=P(A,b,c)\cap\ZZ^n$.

\begin{theorem}\label{parallelep_l2}
Let $N$ be the set of all irreducible points in $M(A,b,c)$, 
where $A\in \ZZ^{n\times n}$ is a non-singular matrix, $c\in\ZZ^n$, $b\in\ZZ^n$, 
then
\begin{equation}
|N|\le 2\prod_{i=1}^{n-1} \left(3+2\log\left(1+\frac{c_i-b_i}{3}\right)\right).
\label{eq:aaN}
\end{equation}
\end{theorem}
\begin{proof}
Denote
\begin{equation}
s_i=\left\lceil\log\left(1+\frac{c_i-b_i}{3}\right)\right\rceil
\qquad
(i=1,2,\dots,n-1).
\label{eq:aasi}
\end{equation}
This implies that
\begin{equation}
3\cdot 2^{s_i - 1} - 3 < c_i - b_i \le 3\cdot 2^{s_i} - 3.
\label{eq:aalogs}
\end{equation}
Denote by $a^{(1)}, a^{(2)}, \dots, a^{(n)}$ the rows of $A$.
Let $j_1,j_2,\dots,j_{n-1}$ be numbers such that $j_i\in\{0,\dots,2s_i\}$ ($i=1,\dots,n-1$).
Let $P(j_1,j_2,\dots,j_{n-1})$ be the set of points in $\RR^n$ that satisfy the following conditions:
for all $i=1,2,\dots,n-1$
\begin{eqnarray}
\begin{array}{cl}
b_i + 2^{j_i} - 1   \le a^{(i)}x < b_i + 2^{j_i + 1} - 1           & (j_i=0,1,\dots,s_i-1),\\[.5em]
b_i + 2^{j_i} - 1 \le a^{(i)}x \le c_i - 2^{j_i} +1                & (j_i=s_i),\\[.5em]
c_i - 2^{j_i -s_i} + 1 < a^{(i)}x \le c_i - 2^{j_i -1 - s_i} + 1   & (j_i=s_i+1,\dots,2s_i),\\[.5em]
\multicolumn{2}{c}{b_n \le a^{(n)}x \le c_n.}
\end{array}
\label{eq:aaparal4}
\end{eqnarray}
It is not hard to see that the set 
of parallelepipeds $P(j_1,j_2,\dots,j_{n-1})$, where $0\le j_i \le 2s_i$ $(i=1,\dots,n-1)$,
is a cover of $P(A,b,c)$.
This is showed on Fig.\,\ref{fig_paral_partition}, where all irreducible points are vertices of the convex hull of $M(A,b,c)$.

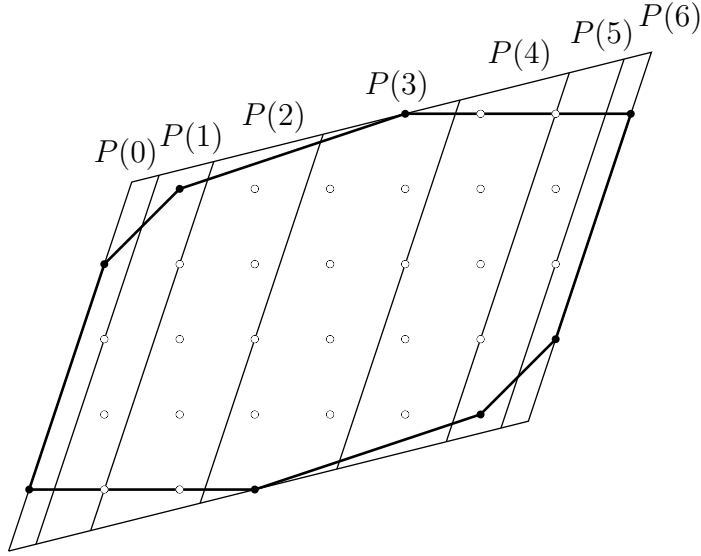
\begin{figure}%
\begin{center}
\psset{xunit=1cm,yunit=1cm}
\begin{pspicture}(0,2)(10,9)

\psline(0.7273,2.1818)(7.6364,3.9091)(9.2727,8.8182)(2.3636,7.0909)(0.7273,2.1818)

\rput(2.3545,7.4636){$P(0)$~}
\rput(3.2273,7.6818){$P(1)$~}
\rput(4.3182,7.9545){$P(2)$~}
\rput(5.9545,8.3636){$P(3)$~}
\rput(7.5909,8.7727){$P(4)$~}
\rput(8.6818,9.0455){$P(5)$~}
\rput(9.5909,9.2727){$P(6)$~}

\psline(1.0909,2.2727)(2.7273,7.1818)
\psline(1.8182,2.4545)(3.4545,7.3636)
\psline(3.2727,2.8182)(4.9091,7.7273)
\psline(5.0909,3.2727)(6.7273,8.1818)
\psline(8.1818,8.5455)(6.5455,3.6364)
\psline(7.2727,3.8182)(8.9091,8.7273)

\psline[linewidth=1pt](1,3)(4,3)(7,4)(8,5)(9,8)(6,8)(3,7)(2,6)(1,3)

\psdots[dotstyle=o]                         (7,8)(8,8)
\psdots[dotstyle=o]          (4,7)(5,7)(6,7)(7,7)(8,7)
\psdots[dotstyle=o]     (3,6)(4,6)(5,6)(6,6)(7,6)(8,6)
\psdots[dotstyle=o](2,5)(3,5)(4,5)(5,5)(6,5)(7,5)
\psdots[dotstyle=o](2,4)(3,4)(4,4)(5,4)(6,4)
\psdots[dotstyle=o](2,3)(3,3)

\psdots(1,3)(4,3)(7,4)(8,5)(9,8)(6,8)(3,7)(2,6)


\end{pspicture}

\caption{Partition of the parallelepiped $0 \le 3x_1 -  x_2 \le 19$, $8 \le -x_1 + 4x_2 \le 26$ into parallelepipeds $P(0),\dots,P(6)$.
In this example all irreducible points are vertices of the convex hull of integer points in the parallelepiped.
Each parallelepiped $P(0)$, $P(3)$, $P(6)$ contains $2$ irreducible points. 
Each $P(1)$, $P(5)$ contains $1$ irreducible point. $P(2)$ and $P(4)$ have no irreducible points.}
\label{fig_paral_partition}

\end{center}
\end{figure}

We show that each $P(j_1,j_2,\dots,j_{n-1})$ contains at most $2$ different points from $N$.
Assume the contrary: let $x,y,z$ be pairwise different points, $x,y,z \in P(j_1,j_2,\dots,j_{n-1})$, and
\begin{equation}
b_n\le a^{(n)}x\le a^{(n)}y\le a^{(n)}z \le c_n.
\label{eq:axayaz}
\end{equation}
We consider the two mutually exclusive cases when 
\begin{equation}
a^{(n)} y - a^{(n)} x \le  a^{(n)} z - a^{(n)} y 
\label{eq:aapar1}
\end{equation}
and when
\begin{equation}
a^{(n)} y - a^{(n)} x >  a^{(n)} z - a^{(n)} y. 
\label{eq:aapar2}
\end{equation}
In the case (\ref{eq:aapar1}) we consider the point $x'=2y-x$, and show that $x'\in P(A,b,c)$.

The conditions $b_n \le a^{(n)}x' \le c_n$ hold, since from (\ref{eq:axayaz}) it follows that
$$
a^{(n)}x' = 2a^{(n)}y - a^{(n)}x \ge a^{(n)}y \ge b_n,
$$
$$
a^{(n)}x' = 2a^{(n)}y - a^{(n)}x \le 2a^{(n)}z - a^{(n)}y \le c_n.
$$

Now we verify the conditions $b_i \le a^{(i)}x' \le c_i$ ($i=1,2,\dots,n-1$).
If $0\le j_i \le s_i-1$, then taking into account (\ref{eq:aalogs}) and (\ref{eq:aaparal4}) we obtain
$$
a^{(i)}x' \le 2(b_i + 2^{j_i+1} - 2) - (b_i + 2^{j_i} - 1) \le b_i + 3\cdot 2^{s_i-1} - 3 < c_i,
$$
$$
a^{(i)}x' \ge 2(b_i + 2^{j_i} - 1) - (b_i + 2^{j_i+1} - 2) = b_i.
$$
If $j_i = s_i$, then
$$
a^{(i)}x' \le 2(c_i - 2^{s_i} + 1) - (b_i + 2^{s_i} - 1) \le c_i,
$$
$$
a^{(i)}x' \ge 2(b_i + 2^{s_i} - 1) - (c_i - 2^{s_i} + 1) \ge b_i.
$$
If $s_i + 1 \le j_i \le 2s_i$, then
$$
a^{(i)}x' \le 2(c_i - 2^{j_i-1-s_i+1} + 1) - (c_i - 2^{j_i-s_i} + 2) = c_i,
$$
$$
a^{(i)}x' \ge 2(c_i - 2^{j_i-s_i} + 2) - (c_i-2^{j_i-1-s_i}+1) \ge c_i + 3 - 3\cdot 2^{s_i-1} > b_i.
$$

Thus, $x'\in M(A,b,c)$ and $y=\frac{1}{2}(x+x')$, hence
$y\notin N$. Contradiction.

In the case (\ref{eq:aapar1}) we consider the point $z'=2y-z$. Using analogous arguments one can show that 
$z'\in M(A,b,c)$ and $y=\frac{1}{2}(z+z')\notin N$.

Thus, each parallelepiped $P(j_1,j_2,\dots,j_{n-1})$ 
contains at most $2$ points from $N$, hence
$|N| \le 2 \prod_{i=1}^{n-1} (1+2s_i)$.
Using (\ref{eq:aasi}) we get (\ref{eq:aaN}).
\end{proof}

\section{Cover of a polytope by parallelepipeds\label{sec_cobering_irr}}

Our method for covering a polytope (i.\,e. bounded convex polyhedron) 
by parallelepipeds consists of constructing a polytope triangulation and 
covering each simplex in the triangulations by parallelepipeds.
Denote
$$
\xi _n (m) = {m-\lfloor \frac{n-1}{2} \rfloor -1 \choose \lfloor \frac{n}{2}
\rfloor }+{m-\lfloor \frac{n}{2} \rfloor -1 \choose \lfloor \frac{n-1}{2}
\rfloor }.
$$

\begin{lemma}\label{l_ChirkovTriang} \cite{Chirkov1993}, see \cite{Shevchenko1995}, \cite{VeselovChirkov2007}
For any $n$-dimensional polytope with $m$ facets ($n$-dimensional faces) there exists its triangulation
with at most $n!\xi_n(m)$ simplexes.
\end{lemma}

The following assertion is a refinement of the result \cite{ChirkovFedotova}.

\begin{lemma} \label{l_ChirkovFedotova} (based upon \cite{ChirkovFedotova})
Any $n$-dimensional simplex $S$ can be covered by at most $(n+1) \cdot {n^2-2\choose n-1}$
$n$-dimensional parallelepipeds.
\end{lemma}

\begin{proof}
Let $v_0, v_1, \dots, v_n$ be the vertices of $S$.
Consider simplexes $S_i = \conv\{ w_{i0}, w_{i1}, \dots, w_{in} \}$ $(i=0,1,\dots, n)$, where
$$
w_{ij} = \frac{1}{n+1} \, v_i + \frac{n}{n+1} \, v_j
\qquad
(i=0,1,\dots, n,~j=0,1,\dots,n)
$$
(in particular, $w_{jj} = v_j$).
It is not hard to verify that $S_0, S_1, \dots, S_n$ form a cover of $S$ (see Fig.\,\ref{fig_simplex_simplex_cover}).

\begin{figure}%
\begin{center}
\psset{xunit=1cm,yunit=0.75cm}
\begin{pspicture}(0,-0.5)(9,6)

{
\psset{hatchwidth=0.5pt}
\pspolygon[fillstyle=vlines,hatchangle=0](0,0)(6,0)(2,4)
\pspolygon[fillstyle=vlines](3,0)(5,4)(9,0)
\pspolygon[fillstyle=hlines](1,2)(7,2)(3,6)
}

\rput[tr](0,0){$v_0$~}
\rput[tl](9,0){~$v_1$}
\rput[br](3,6){$v_2$~}

\end{pspicture}

\caption{The cover of the simplex $S$ by the simplexes $S_0$ (filled by vertical lines), $S_1$ (filled by NW-SE lines) and $S_2$ (filled by SW-NE lines).}
\label{fig_simplex_simplex_cover}

\end{center}
\end{figure}
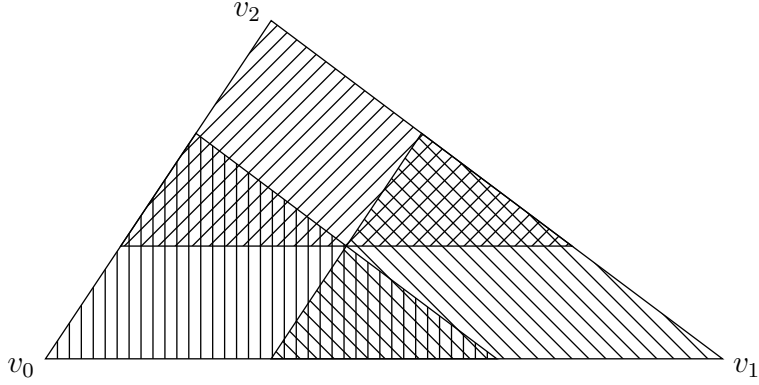

Now, for each $S_i$, we construct parallelepipeds $\Pi_{i1}, \Pi_{i2}, \dots, \Pi_{it}$ that form an {\em outer cover} of $S_i$, i.\,e.
\begin{equation}
S_i \subseteq \bigcup_{k=1}^t \Pi_{ik} \subseteq S \qquad (i=0,1,\dots, n).
\label{eq:SYS0}
\end{equation}
The family of all $\Pi_{ik}$ $(i=0,1,\dots, n,~k=1,2\dots,t)$ will form a cover of $S$.
 
Without loss of generality we suppose that
$$
S = \left\{ x\in\RR^n:~ \sum_{j=1}^n x_j \le n^2 - 1,~ x_j \ge 0 ~~ (j=1,2,\dots,n) \right\}.
$$
Then 
$$
S_0= \left\{x\in\RR^n:~ \sum_{j=1}^n x_j \le n^2 - n,~ x_j \ge 0 ~~ (j=1,2,\dots,n) \right\}.
$$
Let
$$
Y = \left\{ y\in\ZZ^n:~ \sum_{j=1}^n y_j = n^2 - 1,~ y_j \ge 1 ~~ (j=1,2,\dots,n) \right\}.
$$
For each vector $y\in Y$ we consider the parallelepiped
$$
\Pi(y) = \{x\in\RR^n:~ 0\le x_j \le y_j ~~(j=1,2,\dots,n)\}.
$$
It is clear that it is enough to prove (\ref{eq:SYS0}) only for $i=0$.
In our case (\ref{eq:SYS0}) takes the form
\begin{equation}
S_0 \subseteq \bigcup_{y\in Y} \Pi(y)\subseteq S. 
\label{eq:SYS}
\end{equation}
See Fig.\,\ref{fig_simplex_cover}.

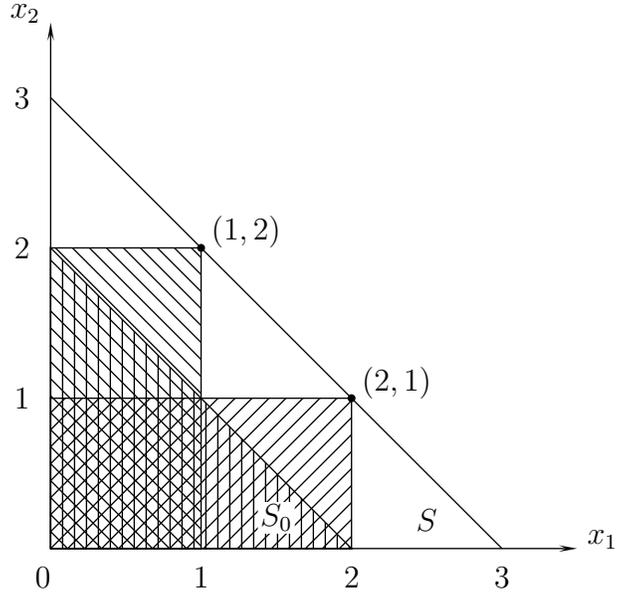
\begin{figure}%
\begin{center}
\psset{xunit=2cm,yunit=2cm}
\begin{pspicture}(0,-0.5)(3.5,3.5)

\psline[arrows=->](0,0)(3.5,0)
\psline[arrows=->](0,0)(0,3.5)
\rput[bl](3.5,0){~$x_1$}
\rput[br](0,3.5){$x_2$~}
\rput[tr](0,-0.12){$0$}
\rput[t](1,-0.12){$1$}
\rput[t](2,-0.12){$2$}
\rput[t](3,-0.12){$3$}
\rput[r](0,1){$1$~~}
\rput[r](0,2){$2$~~}
\rput[r](0,3){$3$~~}

\psline(3,0)(0,3)


{
\psset{hatchwidth=0.5pt}
\pspolygon[fillstyle=vlines,hatchangle=0](0,0)(2,0)(0,2)
\pspolygon[fillstyle=vlines](0,0)(1,0)(1,2)(0,2)
\pspolygon[fillstyle=hlines](0,0)(2,0)(2,1)(0,1)
}

\psdots(2,1)(1,2)
\rput[bl](2,1){~$(2,1)$}
\rput[bl](1,2){~$(1,2)$}

{
\psset{framesep=1pt}
\rput[b](2.5,0.1){\psframebox*{$S$}}
\rput[b](1.5,0.1){\psframebox*{$S_0$}}

}

\end{pspicture}

\caption{$2$-dimensional illustration to (\ref{eq:SYS}).
The simplex $S$ contains the simplex $S_0$ (filled by vertical lines) that is covered by parallelepipeds $\Pi(1,2)$ (filled by NW-SE lines) and $\Pi(2,1)$ (filled by SW-NE lines).}
\label{fig_simplex_cover}

\end{center}
\end{figure}

The second inclusion in (\ref{eq:SYS}) is obvious.
To prove the first one let us consider arbitrary vector $x\in S_0$.
Let $z = \lceil x \rceil$.
Then
$$
\sum_{j=1}^n z_j = \sum_{j=1}^n \lceil x_j \rceil < n^2 - n + n = n^2. 
$$
Taking into account that all components of $z$ are integer we get
$$
\sum_{j=1}^n z_j \le n^2 - 1. 
$$
Increasing (as appropriate) components of $z$ we get a vector $y\in Y$ such that $z \in \Pi (y)$
and consequently $x \in \Pi (y)$.

Thus, we have constructed the family of parallelepipeds $\{\Pi(y):~y\in Y\}$ that satisfy to (\ref{eq:SYS})
and contain a certain vertex (vertex $v_0 = 0$) of the simplex $S$. 
Performing the same construction for every vertex of the simplex we obtain its cover.


Now from (\ref{eq:SYS}) it follows that the constructed family of $(n+1)|Y| = (n+1){n^2-2\choose n-1}$ 
parallelepipeds is a cover of the simplex.
\end{proof}

\begin{lemma}\label{zch_main_l}
Suppose that a polytope $P$ is given as a set of solutions to a system of linear inequalities $Ax\le b$, 
where $A=(a_{ij})\in\ZZ^{m\times n}$, $b=(b_i)\in\ZZ^m$,
$|a_{ij}| \le \alpha$, $|b_i|\le\beta$, then there exists a cover of the polytope by at most
\begin{equation}
\eta_n(m) = n! \,\xi_n(m)\, (n+1)\, {n^2-2\choose n-1}
\label{eq:eta_nm}
\end{equation}
parallelepipeds $\Pi_{\mu}=\{x\in\RR^n:b^{(\mu)}\le A^{(\mu)}x\le c^{(\mu)}\}$ $(\mu=1,\dots,\eta_n(m))$, where
$A^{(\mu)}\in\ZZ^{n\times n}$, $b^{(\mu)}=(b^{(\mu)}_i)\in\ZZ^n$, $c^{(\mu)}=(b^{(\mu)}_i)\in\ZZ^n$,
such that 
\begin{equation}
|c^{(\mu)}_i - b^{(\mu)}_i| \le 2\alpha^{n^2}\beta^n(\sqrt{n})^{n^2+2n+2}.
\label{eq:ci-bi}
\end{equation}
\end{lemma}
\begin{proof}
The required cover is constructed as follows.
First, using Lemma~\ref{l_ChirkovTriang}, we construct the triangulation of the polytope $P$.
Then, using Lemma~\ref{l_ChirkovFedotova}, we construct the cover of each simplex by parallelepipeds.
The upper bound (\ref{eq:eta_nm}) for the total number of parallelepipeds is obtained as a product
of the upper bounds for the number of simplexes in the triangulation and the number of parallelepipeds in the cover of the simplex.

Now we obtain the inequality (\ref{eq:ci-bi}).
First, we find a bound for the quantity of the coefficients in systems of inequalities, which can describe the simplexes in the triangulations.
It is well known that the components of the each vertex $v$ of $P$ (and consequently of simplexes in its triangulation)
can be obtained by turning corresponding $n$ inequalities of $Ax\le b$ into equations. Using Cramer's rule and Hadamard inequality
we get that $v=1/q \cdot (p_1,p_2,\dots,p_n)$, where $p_j\in\ZZ$ $(j=1,2,\dots,n)$, $q\in\ZZ$.
\begin{equation}
|q| \le \alpha^n (\sqrt{n})^n,
\quad
|p_j| \le \alpha^{n-1}\beta (\sqrt{n})^n
\quad
(j=1,2,\dots,n).
\label{eq:qp_i}
\end{equation}
If $v^{(1)},v^{(2)},\dots,v^{(n)}$ are some vertices of the simplex, $v^{(i)}=1/q^{(i)} \cdot (p^{(i)}_1,p^{(i)}_2,\dots,p^{(i)}_n)$,
then the coefficients of the equation $a_1x_1 + a_2x_2 +\dots a_nx_n = a_0$, which describes the hyperplane passing through these vertices,
can be calculated using the following formulas:
$$
a_0 = \det(p_1, p_2,\dots,p_n),
$$
$$
a_j = (-1)^{j+1} \,q_j\, \det(\mbox{\bf 1}, p_1, \dots, p_{i-1}, p_{i+1}, \dots,p_n)
\quad
(j=1,2,\dots,n),
$$
where $\mbox{\bf 1}$ is the column of ones.
From Hadamard inequality, using (\ref{eq:qp_i}), we get
$$
|a_0| \le \left( (\alpha\sqrt{n})^{n-1} \beta\sqrt{n} \cdot \sqrt{n}\right)^n=\alpha^{n(n-1)}\beta^n (\sqrt{n})^{n^2+n},
$$
\begin{equation}
|a_j| \le (\alpha\sqrt{n})^n  \sqrt{n}  \left( (\alpha\sqrt{n})^{n-1} \beta\sqrt{n} \cdot \sqrt{n}\right)^{n-1}
=\alpha^{n^2-n+1}\beta^{n-1} (\sqrt{n})^{n^2+n},
\label{eq:a_jl555}
\end{equation}
which gives bounds for the quantity of the coefficients in systems of inequalities describing the simplexes in the triangulation.

Now we obtain bounds for coefficients $c^{(\mu)}_i$,  $b^{(\mu)}_i$ of system of inequalities
describing parallelepipeds in the cover of simplexes.
Note that the method used in Lemma~\ref{l_ChirkovFedotova} gives parallelepipeds
with facets which are parallel to facets of the corresponding simplexes. 
Hence the coefficients in LHS of equations for these facets (i.e. the coefficients of the matrices $A^{(\mu)}$)
satisfy the inequality (\ref{eq:a_jl555}).
To obtain a bound for $|b^{(\mu)}_i|$, $|c^{(\mu)}_i|$ we put the components of the vertex $v$ of the simplex into the equation of the facet. 
From (\ref{eq:qp_i}) and (\ref{eq:a_jl555}) it follows that
$$
|b^{(\mu)}_i| \le n\cdot \alpha^{n-1}\beta (\sqrt{n})^n \cdot \alpha^{n^2-n+1}\beta^{n-1} (\sqrt{n})^{n^2+n} = \alpha^{n^2}\beta^n(\sqrt{n})^{n^2+2n+2}.
$$
The same inequality holds for $|c^{(\mu)}_i|$, which gives us (\ref{eq:ci-bi}).
\end{proof}

\section{Irreducible points in a polytope\label{sec_polytope_irr}}

Here, using results from two previous sections, we get a bound for the number of irreducible integer points
in a polytope.

\begin{theorem}\label{tNirred0}
Suppose that a polytope $P$ is given as the set of all solutions to a system of linear inequalities $Ax\le b$, where $A=(a_{ij})\in\ZZ^{m\times n}$, $b=(b_i)\in\ZZ^m$, $|a_{ij}| \le \alpha$, $|b_i|\le\beta$.
Let $N$ be the set of all irreducible points in $P\cap\ZZ^n$, then 
\begin{equation}
|N|\le 2 n! \,\xi_n(m)\, (n+1)\, {n^2-2\choose n-1} \left(3+2\log\left(1+\frac{2}{3}\alpha^{n^2}\beta^n(\sqrt{n})^{n^2+2n+2}\right)\right)^{n-1}.
\label{eq:n_irred_points}
\end{equation}
\end{theorem}

\begin{proof}
Using Lemma~\ref{zch_main_l} we construct a cover of $P$ by parallelepipeds.
Obviously, $N$ is contained in the union of the sets of all irreducible integer points in all parallelepipeds.
To bound the number of irreducible points in a parallelepiped we use Theorem~\ref{parallelep_l2}.
Putting (\ref{eq:ci-bi}) in (\ref{eq:aaN}) and multiplying the result by $\eta_n(m)$ from (\ref{eq:eta_nm}), we get (\ref{eq:n_irred_points}).
\end{proof}

\begin{theorem}\label{tNirred}
Suppose $A\in\RR^{m'\times n}$, $b\in\RR^{m'}$, $P=\{x\in\RR^n:~ Ax\le b\}$, with $P\cap \ZZ^n \subseteq E_k^n$.
If $N$ is the set of all irreducible points in $P\cap \ZZ^n$, then for $|N|$
the inequality (\ref{eq:n_irred_points}) holds, where
$m=m'+2n$, 
$$
\alpha \le \frac{(k-1)^{n-1}(n+1)^{\frac{n+1}{2}}}{2^n},
\quad
\beta \le \frac{(k-1)^n(n+2)^{\frac{n+2}{2}}}{2^{n+1}}.
$$
\end{theorem}

\begin{proof}
For the inequality $ax\le a_0$, where $a\in\RR^n$, $a_0\in\RR$,
we consider the system of $k^n+1$ homogeneous linear inequalities in the variables $b_0 \in \RR$, $b\in\RR^n$, $b_{n+1}\in\RR$:
\begin{equation}
\left\{
\begin{array}{rl}
        b_0 - bx \phantom{{}-b_{n+1}} \ge 0\phantom{.}  & \mbox{for all $x$, such that $ax\le a_0$,} \\
       -b_0 + bx            -b_{n+1}  \ge 0\phantom{.}  & \mbox{for all $x$, such that $ax >  a_0$,} \\
                             b_{n+1}  \ge 0. 
\end{array}
\right.
\label{eq:bbetaCone}
\end{equation}
The set $K$ of all its solutions is a polyhedral cone in $\RR^{n+2}$.
Obviously, any vector in $K$, with $b_{n+1} >0$, has components $b_0$, $b$, $b_{n+1}$, such that
$\{x\in E_k^n:~ ax\le a_0\} = \{x\in E_k^n:~ bx\le b_0\}$.

We prove that the cone $K$ is pointed, i.\,e. it does not contain nonzero subspaces.
Suppose that both vectors $\pm(b_0,b,b_{n+1})$ belong to $K$.
In this case, from the last inequality in (\ref{eq:bbetaCone}), we get that $b_{n+1} = 0$,
then, from other inequalities, it follows that $bx=b_0$ for all $x\in E_k^n$.
Since $k\ge 2$, then the affine dimension of $E_k^n$ is $n$, hence $b_0=b=b_{n+1}=0$.
Thus, $K$ does not contain nonzero subspaces.

From the theory of linear inequalities (see, for example, \cite{Schrijver}) it follows that
the set of extreme vectors $g^{(1)},g^{(2)},\dots,g^{(s)}$ of $K$ forms its generative system,
that is, $K=\cone\{g^{(1)},g^{(2)},\dots,g^{(s)}\}$. Moreover, for each $i=1,2,\dots,s$ 
there exists a subsystem of (\ref{eq:bbetaCone}) that becomes a system of equalities on $g^{(i)}$,
with coefficients of the system forming a matrix $T_i$ of rank $n+1$.
This implies that $g^{(i)}$ can be chosen integer with its $j$-th component equal up to sign
to the minor of order $n+1$ cut down from $T_i$ by removing its $j$-th column. 

Let us bound the quantity of the minor.
Multiplying the rows corresponding to $x$ with $ax > a_0$ and columns corresponding to $b$ by $-1$, we get a minor with nonnegative entries.
Using well-known bounds for a determinant with nonnegative entries (see, for example, \cite{Prasolov}), we get the following
bounds for the components of $g^{(i)} = (g^{(i)}_0,g^{(i)}_1,\dots,g^{(i)}_n,g^{(i)}_{n+1})$:
$$
|g^{(i)}_0| \le \frac{(k-1)^n(n+2)^{\frac{n+2}{2}}}{2^{n+1}},
\quad
|g^{(i)}_j| \le \frac{(k-1)^{n-1}(n+1)^{\frac{n+1}{2}}}{2^n}
\quad (j=1,\dots,n).
$$
Among vectors $g^{(1)},g^{(2)},\dots,g^{(s)}$ there is a vector $(b_0,b,b_{n+1})$ with
$b_{n+1} > 0$. We call the inequality $bx\le b_0$ the approximation to the inequality $ax\le a_0$.

We replace all inequalities describing $P$ by those approximations and append $2n$ inequalities $0\le x_j \le k-1$ $(k=1,2,\dots,n)$.
The inequality to be proved follows now from Theorem~\ref{tNirred0}.
\end{proof}

Note that, for any fixed $n$, the bounds for $|N|$ in Theorems~\ref{tNirred0} and~\ref{tNirred} take the form, respectively, 
$$
|N| = O\left(m^{\lfloor\frac{n}{2}\rfloor}\log^{n-1} (\alpha\beta)\right),
\quad
|N| = O(m^{\lfloor\frac{n}{2}\rfloor}\log^{n-1} k).
$$

\section{Bounds for the teaching dimension of threshold functions\label{sec_td_irr}}

Here we prove the hypothesis concerning the teaching dimension of the class of threshold functions of $k$-valued logic in $n$ variables.
Recall that a set $T\subseteq E_k^n$ is called a {\em teaching} set for a threshold function $f\in\TT(n,k)$, 
iff, for any function $g\in\TT(n,k) \setminus\{f\}$, there is a point $z\in T$ such that $f(z)\ne g(z)$.
A teaching set $T$ of $f\in\TT(n,k)$ is called a {\em minimal teaching} set,
if no proper subset of it is teaching for $f$.
A point $z\in E_k^n$ is called {\em essential} for $f$ iff there exists a function $g\in\TT(n,k)$,
such that $f(z)\ne g(z)$ and $f(x) = g(x)$ for all $x\ne z$.
It is known (see, for example, \cite{ShevchenkoZolotykh1998}, \cite{ZolotykhShevchenko1999}), 
that the minimal teaching set $T(f)$ of any threshold function $f$ is unique and consists of all {\em essential points}.
For illustration see Fig.\,\ref{fig_th_f}, where the minimal teaching set for a threshold function $f\in \TT(2,20)$ is drawn.
The maximum cardinality of the minimal teaching set,
$$
\sigma(n,k) = \max_{f\in\TT(n,k)} |T(f)|,
$$
is called the {\em teaching dimension}.
It is known \cite{Hegedus1994}, \cite{ShevchenkoZolotykh1998}, \cite{ZolotykhShevchenko1999} that for any fixed $n\ge 2$
$$
\sigma(n, k) = O(\log^{n-1} k),
\qquad
\sigma(n, k) = \Omega(\log^{n-2} k),
\qquad
\sigma(2,k)=4.
$$
Here we prove that $\sigma(n, k) = \Theta(\log^{n-2} k)$.

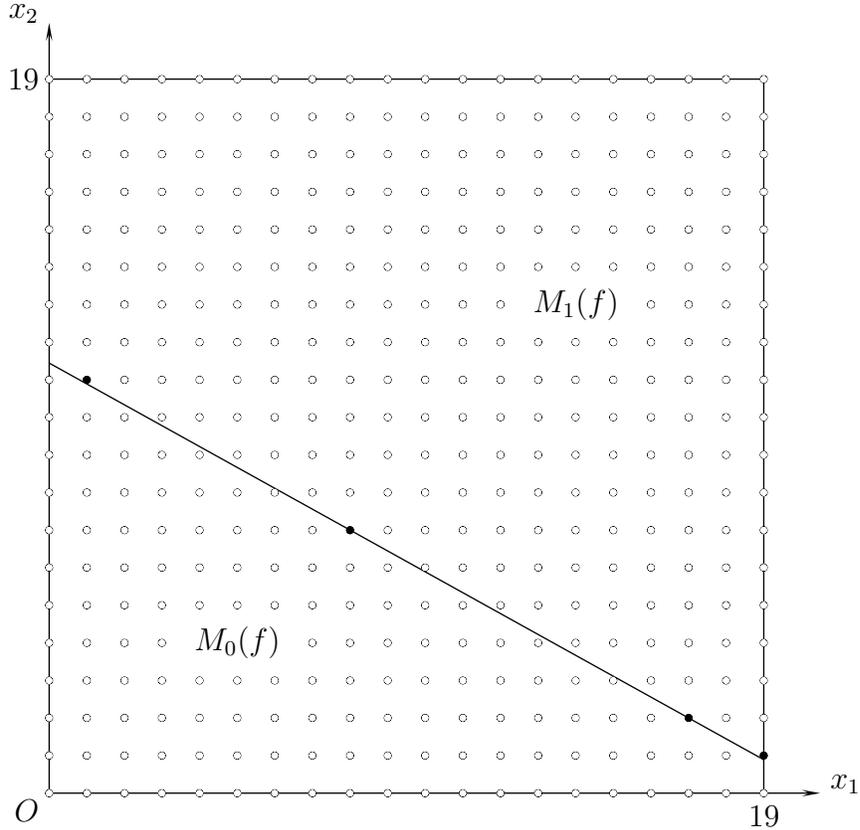
\begin{figure}[tb]
  \centering
\psset{xunit=0.5cm,yunit=0.5cm}
\begin{pspicture}(-.5,-1)(20,20)


  \psline(19,0)(19,19)(0,19)
  
  \psline(0,11.4444)(19,0.8889)
  
  \psline[arrows=->](20.5,0)
  \psline[arrows=->](0,20.5)
  \rput[bl](20.5,0){~$x_1$}
  \rput[br](0,20.5){$x_2$~}
  \rput[tr](0,-0.17){$O$~}
  \rput[t](19,-0.35){$19$}
  \rput[r](0,19){$19$~}

{\psset{dotstyle=o}
  \psdots(0,0)(1,0)(2,0)(3,0)(4,0)(5,0)(6,0)(7,0)(8,0)(9,0)(10,0)(11,0)(12,0)(13,0)(14,0)(15,0)(16,0)(17,0)(18,0)(19,0) 
  \psdots(0,1)(1,1)(2,1)(3,1)(4,1)(5,1)(6,1)(7,1)(8,1)(9,1)(10,1)(11,1)(12,1)(13,1)(14,1)(15,1)(16,1)(17,1)(18,1)(19,1) 
  \psdots(0,2)(1,2)(2,2)(3,2)(4,2)(5,2)(6,2)(7,2)(8,2)(9,2)(10,2)(11,2)(12,2)(13,2)(14,2)(15,2)(16,2)(17,2)(18,2)(19,2) 
  \psdots(0,3)(1,3)(2,3)(3,3)(4,3)(5,3)(6,3)(7,3)(8,3)(9,3)(10,3)(11,3)(12,3)(13,3)(14,3)(15,3)(16,3)(17,3)(18,3)(19,3) 
  \psdots(0,4)(1,4)(2,4)(3,4)(4,4)(5,4)(6,4)(7,4)(8,4)(9,4)(10,4)(11,4)(12,4)(13,4)(14,4)(15,4)(16,4)(17,4)(18,4)(19,4) 
  \psdots(0,5)(1,5)(2,5)(3,5)(4,5)(5,5)(6,5)(7,5)(8,5)(9,5)(10,5)(11,5)(12,5)(13,5)(14,5)(15,5)(16,5)(17,5)(18,5)(19,5) 
  \psdots(0,6)(1,6)(2,6)(3,6)(4,6)(5,6)(6,6)(7,6)(8,6)(9,6)(10,6)(11,6)(12,6)(13,6)(14,6)(15,6)(16,6)(17,6)(18,6)(19,6) 
  \psdots(0,7)(1,7)(2,7)(3,7)(4,7)(5,7)(6,7)(7,7)(8,7)(9,7)(10,7)(11,7)(12,7)(13,7)(14,7)(15,7)(16,7)(17,7)(18,7)(19,7) 
  \psdots(0,8)(1,8)(2,8)(3,8)(4,8)(5,8)(6,8)(7,8)(8,8)(9,8)(10,8)(11,8)(12,8)(13,8)(14,8)(15,8)(16,8)(17,8)(18,8)(19,8) 
  \psdots(0,9)(1,9)(2,9)(3,9)(4,9)(5,9)(6,9)(7,9)(8,9)(9,9)(10,9)(11,9)(12,9)(13,9)(14,9)(15,9)(16,9)(17,9)(18,9)(19,9) 
  \psdots(0,10)(1,10)(2,10)(3,10)(4,10)(5,10)(6,10)(7,10)(8,10)(9,10)(10,10)(11,10)(12,10)(13,10)(14,10)(15,10)(16,10)(17,10)(18,10)(19,10) 
  \psdots(0,11)(1,11)(2,11)(3,11)(4,11)(5,11)(6,11)(7,11)(8,11)(9,11)(10,11)(11,11)(12,11)(13,11)(14,11)(15,11)(16,11)(17,11)(18,11)(19,11) 
  \psdots(0,12)(1,12)(2,12)(3,12)(4,12)(5,12)(6,12)(7,12)(8,12)(9,12)(10,12)(11,12)(12,12)(13,12)(14,12)(15,12)(16,12)(17,12)(18,12)(19,12) 
  \psdots(0,13)(1,13)(2,13)(3,13)(4,13)(5,13)(6,13)(7,13)(8,13)(9,13)(10,13)(11,13)(12,13)(13,13)(14,13)(15,13)(16,13)(17,13)(18,13)(19,13) 
  \psdots(0,14)(1,14)(2,14)(3,14)(4,14)(5,14)(6,14)(7,14)(8,14)(9,14)(10,14)(11,14)(12,14)(13,14)(14,14)(15,14)(16,14)(17,14)(18,14)(19,14) 
  \psdots(0,15)(1,15)(2,15)(3,15)(4,15)(5,15)(6,15)(7,15)(8,15)(9,15)(10,15)(11,15)(12,15)(13,15)(14,15)(15,15)(16,15)(17,15)(18,15)(19,15) 
  \psdots(0,16)(1,16)(2,16)(3,16)(4,16)(5,16)(6,16)(7,16)(8,16)(9,16)(10,16)(11,16)(12,16)(13,16)(14,16)(15,16)(16,16)(17,16)(18,16)(19,16) 
  \psdots(0,17)(1,17)(2,17)(3,17)(4,17)(5,17)(6,17)(7,17)(8,17)(9,17)(10,17)(11,17)(12,17)(13,17)(14,17)(15,17)(16,17)(17,17)(18,17)(19,17) 
  \psdots(0,18)(1,18)(2,18)(3,18)(4,18)(5,18)(6,18)(7,18)(8,18)(9,18)(10,18)(11,18)(12,18)(13,18)(14,18)(15,18)(16,18)(17,18)(18,18)(19,18) 
  \psdots(0,19)(1,19)(2,19)(3,19)(4,19)(5,19)(6,19)(7,19)(8,19)(9,19)(10,19)(11,19)(12,19)(13,19)(14,19)(15,19)(16,19)(17,19)(18,19)(19,19) 
}  
  \psdots(17,2)(8,7)(19,1)(1,11)

  \rput(5,4){\psframebox*{$M_0(f)$}}
  \rput(14,13){\psframebox*{$M_1(f)$}}

\end{pspicture}
  \caption{Threshold function $f\in\TT(2,20)$ with threshold inequality $5x_1 +9x_2\le 103$.
           The minimal teaching set of $f$ consists of $4$ points: $(17,2)$, $(8,7)$ in $M_0(f)$ and $(19,1)$, $(1,11)$ in $M_1(f)$.} 
  \label{fig_th_f}
\end{figure}

Suppose $f\in\TT(n,k)$.
Let $K(f)= \cone\bigl( M_1(f) - M_0(f) \bigr)$, $F_0(f) = \conv  M_0(f) - K(f)$, $F_1(f) = \conv  M_1(f) + K(f)$.

In \cite{ZolotykhChirkov2012} a characterization of $T(f)$ in terms of
$F_0(f)$, $F_1(f)$ is proposed.

Denote $T_{\nu}(f) = T(f)\cap M_{\nu}(f)$ $(\nu=0,1)$.

\begin{theorem} \cite{ZolotykhChirkov2012}
Let $f\in \TT(n, k)$, then $T_{\nu}(f) = \Ver F_{\nu} (f)$ $({\nu}=0,1)$.
\end{theorem}

\begin{corollary}\label{l_razd0}
Let $f\in \TT(n, k)$, $x,y \in T_{\nu}(f)$ $(\nu = 0,1)$, $x\ne y$,
then
\begin{equation}
2x-y\not\in F_0(f)\cup F_1(f). 
\label{eq:F0F1_razd0}
\end{equation}
\end{corollary}

Unfortunately, no convenient description of $F_0(f)\cup F_1(f)$ is known in the general case.
Nevertheless we consider a set $\TT'(n,k)$ of functions $f$, each of which can be given by a threshold inequality such that
$$
a_0\in{\bf Z},\quad a_j\in{\bf Z}, \quad 0< a_0 < a_j(k-1) \qquad (j=1,2,\dots,n).   
$$
Denote by $\ZZ_+^n$ the set of all vectors in $\ZZ^n$ with nonnegative components.
We say that a set $G\subset \ZZ_+^n$ has {\em the separation property}, iff 
from conditions $x,y\in G$, $x\ne y$ it follows that $2x-y\notin\ZZ_+^n$ \cite{Shevchenko1981}.
One can verify \cite{ZolotykhChirkov2012} that if $f\in \TT'(n,k)$, then $F_0(f)\cup F_1(f) = {\bf Z}_+^n$ and, consequently,
the property (\ref{eq:F0F1_razd0}) is equivalent to the separation property. From this we get the following result.

\begin{theorem}\label{t_ZolotykhChirkov2012-} \cite{ZolotykhChirkov2012}
If $f\in\TT'(n,k)$ and $n\ge 2$, then 
$$
|T_{\nu}(f)| \le n(1 + \log n)\Bigl( 1 + \log (k+1)  \Bigr)^{n-2} \qquad (\nu=0,1). 
$$
\end{theorem}

\begin{theorem}\label{th_upper_sigma}
For any fixed $n\ge 2$
$$
\sigma(n,k) = O(\log^{n-2} k)
\qquad
(k\to\infty).
$$ 
\end{theorem}

\begin{proof}
Without loss of generality, we suppose that the coefficients of the threshold inequality of the function $f\in\TT(n,k)$
satisfy the conditions $a_1\ge a_2 \ge\dots \ge a_n\ge 0$.

If $a_0\le (k-1)a_n$, then the bound to be proved follows from Theorem~\ref{t_ZolotykhChirkov2012-}.
Now we consider the case when $a_0> (k-1)a_n$. If $e_n\notin K(f)$, then form $x\in T_{\nu}(f)$ 
it follows that $x_n=0$ or $x_n=k-1$, hence $|T_{\nu}(f)|\le 2\sigma(n-1,k)$. 

We denote by $e_j$ the vector with all components equal to $0$ except the $j$-th component equal to $1$.
Suppose $e_n\in K(f)$. Let
$$
T'_0(f)=\left\{x\in T_0(f):\ \sum_{j=1}^{n-1}a_jx_j\le a_0-(k-1)a_n\right\},
$$ 
$$
T''_0(f)=\left\{x\in T_0(f):\ \sum_{j=1}^{n-1}a_jx_j> a_0-(k-1)a_n\right\},
$$
$$
T'_1(f)=\left\{x\in T_1(f):\ \sum_{j=1}^{n-1}a_jx_j> a_0\right\},
$$
$$
T''_1(f)=\left\{x\in T_1(f):\ \sum_{j=1}^{n-1}a_jx_j\le a_0\right\}.
$$ 
If $x\in T'_{\nu}(f)$ ($\nu = 0,1$), then $x_n=0$ or $x_n=k-1$, hence $|T'_{\nu}(f)|\le 2\sigma(n-1,k)$. 

Let 
$$
P=\left\{x\in E_k^n:\ a_0-(k-1)a_n <\sum_{j=1}^{n-1}a_jx_j\le a_0,\ x_n=0\right\}.
$$ 
If $y\in P$, then, taking into account that $e_n\in K(f)$, we get
$$
\{y+\alpha e_n:~ \alpha\in Z\} \subset F_0(f)\cup F_1(f).
$$
Therefore, using Lemma~\ref{l_razd0}, we get that $|T''_{\nu}(f)|$ does not exceed 
the number of irreducible points in $P$. Since the dimension of $\conv P$ is at most 
$n-1$, then, by Theorem~\ref{tNirred}, the number of irreducible points
in $P$ is $O(\log^{n-2}k)$ when $n$ is fixed.
\end{proof}

Taking into account the lower bound $\sigma(n,k) = \Omega(\log^{n-2} k)$ (for fixed $n\ge 2$), 
obtained in \cite{ShevchenkoZolotykh1998}, \cite{ZolotykhShevchenko1999},
from Theorem~\ref{th_upper_sigma} we get the following assertion.
\begin{corollary}\label{cor_main_logn-2}
For any fixed $n\ge 2$
$$
\sigma(n,k) = \Theta(\log^{n-2} k)
\quad
(k\to\infty).
$$
\end{corollary}

\paragraph{Acknowledgments} The authors thank V.\,N.\,Shevchenko and S.\,I.\,Veselov for fruitfull discussions
and referees for usefull suggestions.


\begin{thebibliography}{99}

\bibitem{ABS1995}
    {\bf Antony~M., Brightwell~G., Shawe-Taylor~J.}
    {\it On exact specification by labelled examples}.
    Discrete Applied Mathematics. 
    61 (1), 1995, 1--25.
    
\bibitem{BHL1992}
{\bf B\'ar\'any~I., Howe~R., Lov\'asz~L.}
{\it On integer points in polyhedra: a lower bound}. 
Combinatorica.
12 (2), 135--142, 1992.

\bibitem{Chirkov1993}
{\bf Chirkov~A.\,Yu.} 
{\it Caratheodory’s theorem and coverings of a polyhedron by simplexes}.
Manuscript No. 668–B93, deposited at VINITI, Moscow, 1993. (Russian)


\bibitem{Chirkov1996}
{\bf Chirkov~A.\,Yu.}
{\it On the lower bound for the number of vertices of
convex hull of integer and partially integer points of a polyhedron}.
Discrete analysis and operation research. 
3 (2), 1996, 80--89.


\bibitem{Chirkov1997}
{\bf Chirkov A. Yu.}
{\it The relationship between upper bounds of the
number of vertices of convex hull of integer points of a polyhedron
and its metric characteristics}. 
Proceedings of the First
International Conference “Mathematical Algorithms”.
Nizhni Novgorod State University Publisher,
1997, 169--174. (Russian) 

\bibitem{ChirkovFedotova}
{\bf Chirkov~A.\,Yu., Fedotova~A.\,A.} 
{\it On coverings of a polyhedron by parallelepipeds}.
Manuscript No. 1361-В94, deposited at VINITI, Moscow, 1994. (Russian)



\bibitem{CHKMcD1992}
{\bf Cook~W., Hartmann~M., Kannan~R., McDiarmid~C.} 
{\it On integer points in polyhedra}.
Combinatorica. 12 (1), 1992, 27--37.

\bibitem{HayesLarman1983}
{\bf Hayes~A.\,S., Larman~D.\,C.} 
{\it The vertices of the knapsack polytope}.
Discrete Applied Mathematics.
6 (2), 1983, 135--138.


\bibitem{Hegedus1994} 
{\bf Heged\"us~T.} 
{\it Geometrical concept learning and convex polytopes}.
{Proc. 7th Ann. ACM Conf. on
Computational Learning Theory (COLT'94).} 
New York: ACM Press, 1994, 228--236.

\bibitem{Hegedus1995}
{\bf Heged\"us~T.} 
{\it Generalized teaching dimensions and the query complexity of learning}.
{Proc. 8th Ann. ACM Conf. on
Computational Learning Theory (COLT'95).} 
New York: ACM Press, 1995, 108--117.

\bibitem{Morgan1991}
{\bf Morgan~D.\,A.}
{\it Upper and lower bound results on the convex hull of integer points in polyhedra}. 
Mathematika, 38 (2), 1991, 321--328.

\bibitem{Prasolov} 
{\bf Prasolov~V.\,V.} 
{\it Problems and Theorems in Linear Algebra}. 
Translations of Mathematical Monographs.
134.
AMS, Providens, Rhode Island, 1994.


\bibitem{Schrijver}  
{\bf Schrijver~A.}
{\it Theory of Linear and Integer Programming}. Wiley--Interscience New York, 1986.



\bibitem{Shevchenko1981}
{\bf Shevchenko~V.\,N.} 
{\it On the number of extreme points in integer programming}. 
Kibernetika, (2), 1981, 133--134.

\bibitem{Shevchenko1985} 
{\bf Shevchenko~V.\,N.}
{\it On some functions of many--valued logic connected with integer programming}.
Methods of Discrete Analysis in the Theory of Graphs and Circuits.
Novosibirsk, 42, 1985, 99--102. (Russian)

\bibitem{Shevchenko1995}
{\bf Shevchenko~V.\,N.} 
{\it Qualitative Topics in Integer Linear Programming}.
Translations of Mathematical Monographs.
156, AMS, Providens, Rhode Island, 1997.


\bibitem{ShevchenkoZolotykh1998}
{\bf Shevchenko~V.\,N., Zolotykh~N.\,Yu.} 
{\it On complexity of deciphering threshold functions of $k$-valued logic}.
Russian Mathematical Doklady.
362 (5), 1998, 606--608.

\bibitem{ShevchenkoZolotykh1998a}
{\bf Shevchenko~V.\,N., Zolotykh~N.\,Yu.} 
{\it Lower bounds for the complexity of learning half-spaces with membership queries}.
Algorithmic Learning Theory, Lecture Notes in Artificial Intelligence.
1501, 1998, 61--71.

\bibitem{Veselov1984}
{\bf Veselov~S.\,I.}
{\it A lower bound for the mean number of irreducible and
extreme points in two discrete programming problems}. 
Manuscript
619–84, deposited at VINITI, Moscow, 1984. (Russian)
    
    
\bibitem{VeselovChirkov2007}  
{\bf Veselov~S.\,I., Chirkov~A.\,Yu.}
{\it Some estimates for the number of vertices of integer polyhedra}.
Journal of Applied and Industrial Mathematics. 
2 (4), 2008, 591--604.

\bibitem{VirovlyanskayaZolotykh2003}
{\bf Virovlyanskaya~M.\,A., Zolotykh~N.\,Yu.}
{\it An upper bound for the mean cardinality of minimal teaching set of threshold function of many-valued logic}.
Vestnik of University of Nizhni Novgorod. Mathematical modeling and optimal control.
2003, 238--246. (Russian)




\bibitem{Zolotykh2006}
{\bf Zolotykh~N.\,Yu.}	
{\it On the number of vertices in integer linear programming problems}.
	arXiv:math/0611356 [math.CO], 2006.

\bibitem{Zolotykh2008}
{\bf Zolotykh~N.\,Yu.} 
{\it Bounds for the cardinality of the minimal teaching set of a threshold function of many-valued logic}.
Mathematical Topics in Cybernatics (17).
Moscow, Fizmatlit, 2008,
159--168.
(Russian)


\bibitem{ZolotykhChirkov2012}
{\bf Zolotykh~N.\,Yu., Chirkov~A.\,Yu.} 
{\it On the upper bound for cardinality of the minimal teaching set of a threshold function}.
Discrete Analysis and Operations Research. 
19 (5), 2012, 35--46. (Russian)


\bibitem{ZolotykhShevchenko1995}  
{\bf Zolotykh~N.\,Yu., Shevchenko~V.\,N.} 
{\it On complexity of deciphering threshold functions}.
Discrete Analysis and Operations Research.
2 (3), 1995, 72--73. (Russian)
 


\bibitem{ZolotykhShevchenko1999}
{\bf Zolotykh~N.\,Yu., Shevchenko~V.\,N.} 
{\it Estimating the complexity of deciphering a threshold functions in a $k$-valued logic}.
Computational Mathematics and Mathematical Physics.
39 (2), 1999, 328--334.


    


\end{thebibliography}
\end{document}